\documentclass[11pt,reqno]{amsart}

\usepackage{fullpage, setspace}
\usepackage{amsmath, amssymb, amsthm}
\usepackage{float, graphicx, subfigure}
\usepackage{amsfonts, amstext, bbm}


\newtheorem{theorem}{Theorem}
\newtheorem{lemma}[theorem]{Lemma}
\newtheorem{proposition}[theorem]{Proposition}

\newtheorem{corollary}[theorem]{Corollary}
\newtheorem{definition}[theorem]{Definition}

{\bf}{\rm}

\newcommand\norm[1]{\left\lVert#1\right\rVert}
\newcommand{\PR}{\mathbb P}

\newcommand{\ER}{\mathbb E}

\newcommand{\Xc}{\mathcal X}
\newcommand{\mytilde}{\raise.17ex\hbox{$\scriptstyle\mathtt{\sim}$}}

\begin{document}

\title{Mixing time of vertex-weighted exponential\\ random graphs}

\author{Ryan DeMuse \and Terry Easlick \and Mei Yin}

\thanks{Mei Yin's research was partially supported by NSF grant DMS-1308333.}

\address{Department of Mathematics, University of Denver, Denver, CO 80208,
USA} \email {ryan.demuse@du.edu \\ terry.easlick@du.edu \\ mei.yin@du.edu}

\dedicatory{\rm \today}

\begin{abstract}
Exponential random graph models have become increasingly important in the study of modern networks ranging from social networks, economic networks, to biological networks. They seek to capture a wide variety of common network tendencies such as connectivity and reciprocity through local graph properties. Sampling from these exponential distributions is crucial for parameter estimation, hypothesis testing, as well as understanding the features of the network in question. We inspect the efficiency of a popular sampling technique, the Glauber dynamics, for vertex-weighted exponential random graphs. Letting $n$ be the number of vertices in the graph, we identify a region in the parameter space where the mixing time for the Glauber dynamics is $\Theta(n \log n)$ (the high temperature phase) and a complement region where the mixing time is exponentially slow on the order of $e^{\Omega(n)}$ (the low temperature phase). Lastly, we give evidence that along a critical curve in the parameter space the mixing time is $O(n^{2/3})$.

\vskip.1truein


\vskip.1truein

\end{abstract}

\maketitle

\newpage

\section{Introduction}
\label{intro}
Exponential random graph models are powerful tools in the study of modern networks ranging from social networks, economic networks, to biological networks. By representing the complex global structure of a large network through tractable local properties, these models seek to capture a wide variety of common network tendencies. See for example Bollob\'{a}s \cite{B}, Durrett \cite{Durrett}, van der Hofstad \cite{Hofstad}, Newman \cite{Newman.book}, and references therein. Despite their flexibility, conventionally used exponential random graphs suffer from some deficiencies that may hamper their utility to researchers \cite{CD} \cite{Krioukov}. Consider the dynamics of spreading events in a complex network. There are sensitive control points collectively known as ``influential spreaders'', whose infection maximizes the overall fraction of infected vertices. For information diffusion over Twitter for example, the influential spreaders may be thought of as a celebrity or a news source. Vertices in the network thus carry with themselves some distinguishing features, a phenomenon that could not be directly modeled by standard exponential random graphs since their underlying probability space consists of simple graphs only. By placing weights on the vertices, the current work addresses this limitation of the exponential model.

Before proceeding further, we provide another reason why the vertex-weighted model may be of interest \cite{BR}. Let $\mathcal{G}_n$ be the set of all vertex-weighted labeled graphs $G_n$ on $n$ vertices. Assume that the vertex weights are iid real random variables subject to a common distribution $\nu$ supported on $[0, 1]$, the edge weight between two vertices is a product of the vertex weights, and the triangle weight among three vertices is a product of the edge weights. Let $U$ be a random variable distributed according to $\nu$ and denote its expectation with respect to $\nu$ by $\ER$. Further denote the expected edge weight of $G_n \in \mathcal{G}_n$ by $e$ and the expected triangle weight by $t$. Then we have $e=\ER(U)^2$ and $t=(\ER(U^2))^3$.
Note that by suitably choosing $\nu$, the entire region between the upper boundary of the realizable edge-triangle densities and the Erd\H{o}s-R\'{e}nyi curve ($e^3 \le t \le e^{3/2}$) may be attained for vertex-weighted random graphs. See Figure \ref{et}. If we take $U$ to be Bernoulli, the
upper boundary is reproduced:
\begin{equation}
\ER(U) \ge \ER(U^2) \implies e^{3/2} \ge t.
\end{equation}
If we take $U$ to be a constant a.s., the
Erd\H{o}s-R\'{e}nyi curve is recovered:
\begin{equation}
\ER(U^2) \ge \ER(U)^2 \implies t \ge e^3.
\end{equation}
By contrast, since simple graphs may be interpreted as having iid Bernoulli$(.5)$ weights on the edges, the underlying graph space of standard exponential random graphs lies at a single point $(1/2, 1/8)$ on the Erd\H{o}s-R\'{e}nyi curve. Even without incorporating the exponential construction, assigning vertex weights alone adds intriguing characteristics to the model.

\begin{figure}[t!]
	\centering
	\includegraphics[clip=true, height=2.5in]{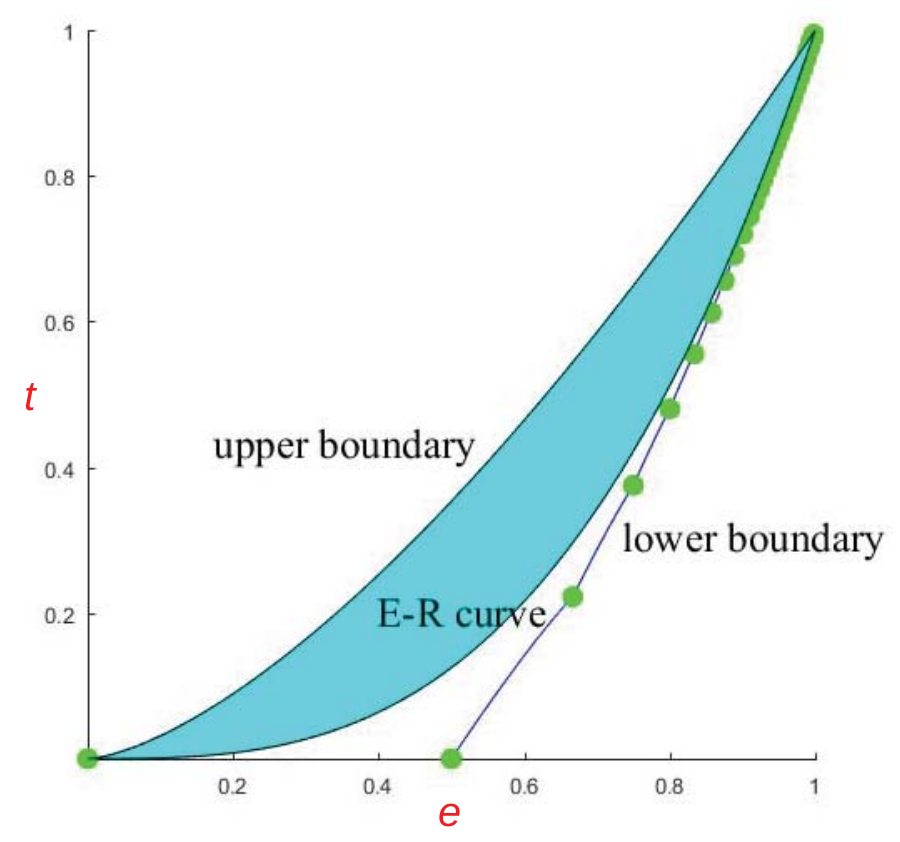}
	\caption{The cyan region shows where the expected edge and triangle densities lie for a vertex-weighted random graph model.}
	\label{et}
\end{figure}

\subsection{The model}
\label{model}
In this paper we will restrict our attention to vertex-weighted exponential random graph models where the vertex weights take values in $0$ and $1$ only. We include in the exponent a combination of edge and triangle densities, both with non-negative parameters. Even under this simplification, the vertex-weighted model depicts captivating behaviors in large-scale networks. Instead of Erd\H{o}s-R\'{e}nyi, it emphasizes the formation of cliques, and is particularly suited for the modeling of a broad range of social networks. Consider the Facebook friend graph for example, where we make the idealistic assumption that a person is either interested (vertex value $1$) or not (vertex value $0$) in building a friendship. Then having value $1$ at three distinct vertices will force the formation of a triangle rather than a two-star, which is in analogy with the common conception that a friend of a friend is more likely to be a friend. An added benefit of this setting is that the model may be considered as an extension of the lattice gas (Ising) model on a graph, and the techniques of spin models may be employed in our investigation \cite{EFHR}.

A graph $G_n \in \mathcal{G}_n$ may be viewed as an element of $X \in \Xc:=\{0,1\}^n$, referred to as ``configurations'', that attributes weights $0$ or $1$ to the ordered vertices of $G_n$. Denote by $X(i)$ the weight of vertex $i$. Borrowing terminology from spin models, the vertex weight $X(i)$ will be called the ``spin'' at $i$. By iid-ness, the spins at different vertices are independent, and subject to a common distribution $\nu$ that assumes value $0$ with probability $1-p$ and $1$ with probability $p$ for some $p \in (0,1)$. Let $H_1$ be the number of edges for the configuration $X$ and $H_2$ be the number of triangles. They may be formulated explicitly in terms of lattice gas (Ising) spins:
\begin{equation}
H_1(X) = \sum_{i \neq j} X(i) X(j)  \text{   and   } H_2(X) = \sum_{i \neq j \neq k} X(i) X(j) X(k),
\end{equation}
where the inequality $n_1 \neq n_2 \neq \cdots \neq n_k$ means that $n_i \neq n_j$ for any $i \neq j$. We rescale the edge and triangle parameters in the exponent, $H=(\alpha_1/n)H_1+(\alpha_2/n^2)H_2$, so that the total contribution of a single vertex to the weights is $O(1)$. We are now ready to introduce a Gibbs distribution to the set of spin configurations $\Xc$. To avoid cumbersome notation, we suppress the $n$-dependence in many of the quantities under discussion.

\begin{definition}
\label{def}
Take $\alpha_1\geq 0$, $\alpha_2\geq 0$, and $p \in (0,1)$. Let $X \in \Xc$ be a spin configuration. Denote by $\omega(X)$ the number of vertices with spin $1$ in $X$. Assign a Gibbs probability measure on $\Xc$ as
	\begin{equation} \label{dist}
		\pi (X) = Z^{-1} \exp(H(X)) p^{\omega(X)} (1-p)^{n-\omega(X)},
	\end{equation}
where $H=(\alpha_1/n)H_1+(\alpha_2/n^2)H_2$ is the combination of edge and triangle weights and $Z = Z(n,p,\alpha_1,\alpha_2)$ is the normalizing constant (partition function).
\end{definition}

The configuration space $\Xc$ can be partially ordered in the sense that for $X, Y \in \Xc$, we say that $X \le Y$ if and only if $X(i) \le Y(i)$ for every $i \in \{1,\dots,n\}$. To model the evolution of the network towards equilibrium, we will adopt (single-site) Glauber dynamics, which is a discrete-time irreducible and aperiodic Markov chain $\left(X_t\right)_{t=0}^\infty$ on $\Xc$. Under the Glauber dynamics, the random graph evolves by selecting a vertex $i$ at random and updating the spin $X(i)$ according to $\pi$ conditioned to agree with the spins at all vertices not equal to $i$. By sampling
from the exponential distribution using Glauber dynamics, we learn the global structure of the network as well as
parameters describing the interactions. Explicitly, let $X \in \Xc$ be a configuration and set the initial state $X_0=X$. The next step of the Markov chain, $X_1$, is obtained as follows. Choose a vertex $i$ uniformly at random from $\{1,\dots,n\}$. Let $X_1(j) = X(j)$ for all $j \neq i$, $X_1(i) = 1$ with probability $P_{+}$ and $X_1(i) = 0$ with probability $P_{-}$, where the update probabilities $P_{+}$ and $P_{-}$ are given by
\begin{equation} \label{upprob}
		P_{+}(X, i) =  \frac{p\exp(H'(X, i))}{p\exp(H'(X, i)) + (1-p)}
\end{equation}
and
\begin{equation} \label{downprob}
		P_{-}(X, i) =  \frac{1-p}{p\exp(H'(X, i)) + (1-p)}.
\end{equation}
Here $H'=(\alpha_1/n)S+(\alpha_2/n^2)T$ depends only on the spins at vertices other than $i$, with
\begin{equation}
\label{ST}
S(X, i) = \sum_{i \neq j} X(j)  \text{   and   } T(X, i) = \sum_{i \neq j \neq k} X(j) X(k)=\frac{S(X, i)(S(X, i)-1)}{2}.
\end{equation}
For $X, Y \in \Xc$, the transition matrix for the Glauber dynamics is then
	\begin{equation} \label{transmat}
		P(X,Y) = \frac{1}{n} \sum_{i} \frac{f(Y(i)) \exp\left( Y(i) H'(X, i)\right)}{f(Y(i)) \exp\left( Y(i) H'(X, i) \right) + f(1-Y(i)) \exp\left( (1-Y(i)) H'(X, i) \right) }\mathbf{1}_{\left\{\substack{ Y(j) = X(j)\\ j \neq i }\right\}},
	\end{equation}
where $Y(i) \in \{0,1\}$ and we define $f$ such that $f(0) = 1-p$ and $f(1) = p$ to lighten the notation.

\subsection{Mixing time}
The Gibbs distribution $\pi$ is stationary and reversible for the Glauber dynamics chain. By the convergence theorem for ergodic Markov chains, the Glauber dynamics will converge to the stationary distribution and our goal is to obtain some estimates on the mixing time, since it greatly affects the efficiency of simulation studies and sampling algorithms \cite{BBS} \cite{CD1}. Given $\varepsilon > 0$, the mixing time for this Markov chain is defined as
	\begin{equation} \label{tmix}
		t_{\text{mix}}(\varepsilon) := \min\{ t : d(t) \le \varepsilon \},
	\end{equation}
where
\begin{equation}
\label{d}
d(t) = \max_{X \in \Xc} \norm{P^{t}(X,{\cdot}) - \pi}_{\text{TV}}
\end{equation}
measures the total variation distance to stationarity of the Glauber dynamics chain after $t$ steps. As is standard, we take $t_{\text{mix}}:= t_{\text{mix}}(1/4)$. The mixing time is thus defined to be the minimum number of discrete time steps such that, starting from an arbitrary configuration $X$, the chain is within total variation distance $1/4$ from the stationary distribution $\pi$. For background on mixing times, see Aldous and Fill \cite{AF} and Levin et al. \cite{MCMT}. Our results will indicate that the mixing time can vary enormously depending on the choice of parameters.

\subsection{Normalized magnetization}
\label{mag}
Given a spin configuration $X \in \Xc$, the normalized magnetization $c$ of $X$ is defined as
\begin{equation}
c(X) = \frac{1}{n} \sum_{i = 1}^{n} X(i).
\end{equation}
Adopting (single-site) Glauber dynamics on $\Xc$, the normalized magnetization chain $\left(c_t\right)_{t=0}^\infty$ is a projection of the chain $\left(X_t\right)_{t=0}^\infty$ and so is also aperiodic and irreducible. Set the initial state $c_0=c$. From the mechanism described in Section \ref{model}, after one Glauber update, $c_1$ will take on one of three values: $c-1/n$, $c$, or $c+1/n$. If a spin $0$ vertex is chosen and updated to spin $1$, $c$ changes to $c + 1/n$. Alternatively, if a spin $1$ vertex is chosen and updated to spin $0$, $c$ changes to $c - 1/n$. When no spins are updated, $c$ stays the same. By (\ref{upprob}), the probability that we select a spin $0$ vertex and update it to spin $1$ is
	\begin{equation} \label{Pu}
		P_u = \frac{ n-cn }{ n } \frac{ p \exp\left( \alpha_1 c + \frac{\alpha_2}{2} c(c - \frac{1}{n}) \right) }{ p \exp\left( \alpha_1 c + \frac{\alpha_2}{2} c(c - \frac{1}{n}) \right)+(1-p) }.
	\end{equation}
Similarly, by (\ref{downprob}), the probability that we select a spin $1$ vertex and update it to spin $0$ is
	\begin{equation} \label{Pd}
		P_d = \frac{ cn }{ n } \frac{ 1-p }{  p\exp\left( \alpha_1 (c - \frac{1}{n}) + \frac{\alpha_2}{2} (c - \frac{2}{n})(c - \frac{1}{n}) \right)+(1-p) }.
	\end{equation}
Combining (\ref{Pu}) and (\ref{Pd}), the magnetization $c_t$ moves up with probability $P_u$, down with probability $P_d$, and remains unchanged with probability $1 - P_u - P_d$. For $n$ large enough, $P_u \asymp (1-c)\lambda(c)$ and $P_d \asymp c(1-\lambda(c))$, where
\begin{equation}
\label{lamb}
\lambda(c) = \frac{ p \exp\left( \alpha_1 c + \frac{\alpha_2}{2} c^2 \right) }{ p\exp\left( \alpha_1 c + \frac{\alpha_2}{2} c^2 \right)+(1-p)}
\end{equation}
represents the asymptotic probability that a chosen vertex is updated to spin $1$. This implies that the expected magnetization drift is asymptotically $(\lambda(c)-c)/n$, and a rigorous justification may be found in Lemma \ref{expmagnetization}.

\begin{figure}[t!]
	\centering
	\includegraphics[clip=true, height=4in, angle=-90]{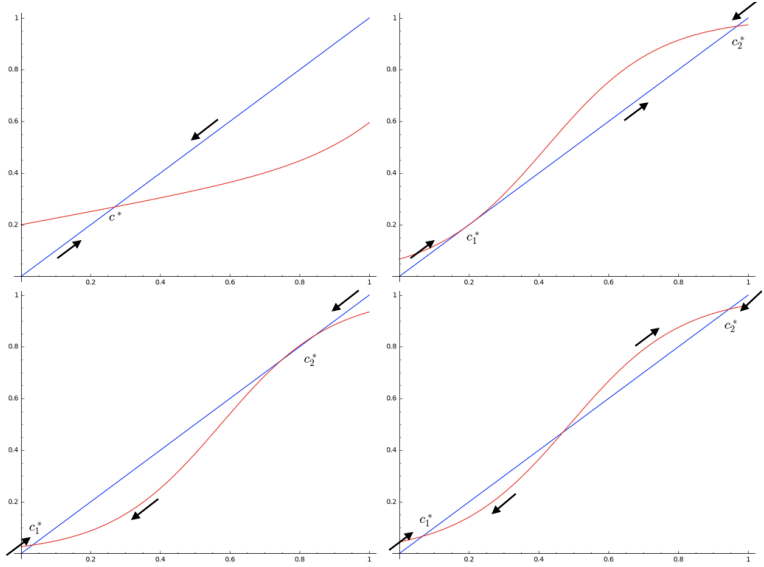}
	\caption{Behavior of the $\lambda$ function in different regions of the parameter space with arrows indicating whether the fixed point is an attractor or a repellor.}
    \label{lambda(c)}
\end{figure}

\subsection{Phase classification}
\label{phase}
The magnetization chain $\left(c_t\right)$ is a deciding factor in the convergence of the Glauber dynamics chain $\left(X_t\right)$. Note that $0 \le c_t \le 1$ and $\lambda$ is a smooth and increasing function on $[0,1]$. Since $\lambda(0)>0$ and $\lambda(1) < 1$, $\lambda(c) = c$ admits at least one solution in $(0, 1)$. If the solution $c$ is unique and not an inflection point, i.e. $\lambda'(c)<1$ (referred to as the ``high temperature phase''), then independent of the initial position all configurations will be driven towards it, and the burn-in stage will cost $O(n)$ steps. See the upper left plot of Figure \ref{lambda(c)}. Conversely, if there exist at least two solutions $c$ such that $\lambda'(c)<1$ (referred to as the ``low temperature phase''), then the burn-in procedure will take the configurations to different attractor states depending on their initial positions. See the lower right plot of Figure \ref{lambda(c)}. Once the configuration is close to an attractor, the Glauber dynamics allows an exponentially small flow of probability for it to leave. A detailed examination of the burn-in period will be provided in Section \ref{burn}.

In Section \ref{fast}, through estimating the average distance after one update between two coupled configurations that agree everywhere except at a single vertex, we show that the Glauber dynamics $X_t$ mixes in $O(n \log n)$ steps in the high temperature phase. Relating to coupon collecting and employing spectral methods, the same asymptotic lower bound $\Omega(n \log n)$ is validated. While in Section \ref{slow}, by a conductance argument using the Cheeger inequality, we establish exponentially slow mixing of the Glauber dynamics $X_t$ in the low temperature phase. Finally, in Section \ref{critical}, we give evidence that the burn-in will cost $O(n^{3/2})$ steps along the ``critical curve'', and the Glauber dynamics $X_t$ is thus expected to mix in $O(n^{3/2})$ steps. See Figure \ref{3D}. The cyan and yellow surfaces separate the high and low temperature phases, with their intersection marked by the critical curve. Convergence of the Glauber dynamics $X_t$ elsewhere on the two surfaces corresponds to the situation where $\lambda(c)=c$ has at least two solutions and one solution $c$ satisfies $\lambda'(c)=1$. See the upper and right and lower left plots of Figure \ref{lambda(c)}. The mixing time largely depends on the movement of the chain around the inflection point and is not addressed in this paper.

\begin{figure}[t!]
	\centering
	\includegraphics[clip=true, width=5.5in]{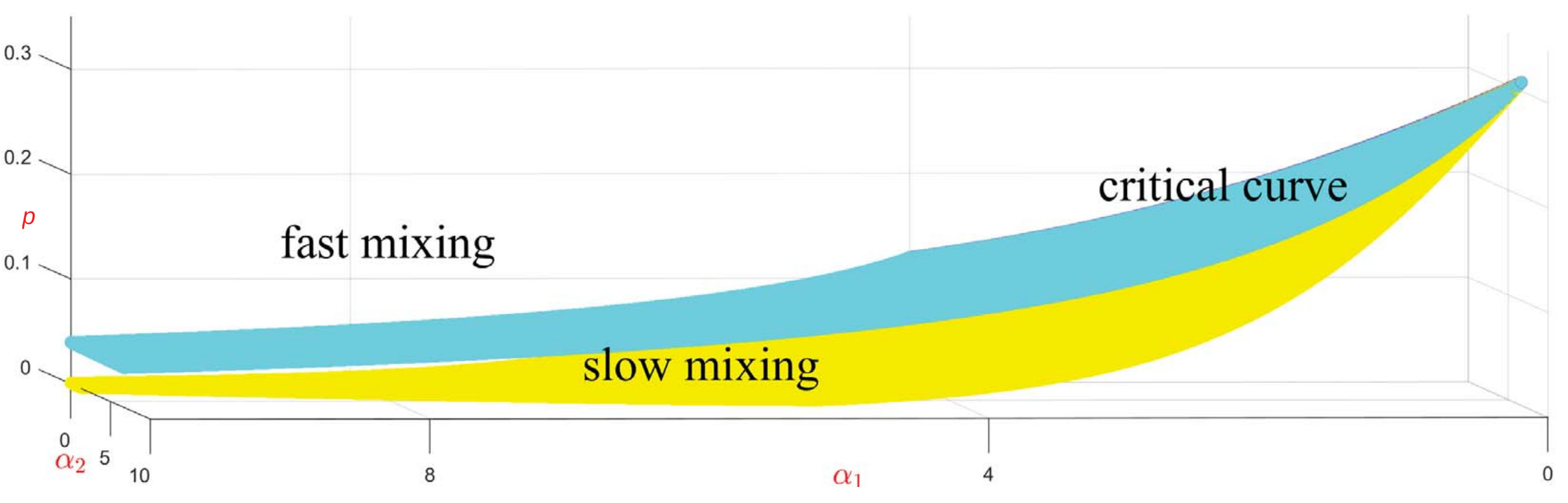}
	\caption{Surfaces in the parameter space illustrating the region with fast vs. slow mixing and identifying the critical curve.}
    \label{3D}
\end{figure}

\section{Burn-in period}
\label{burn}
We start by running the Glauber dynamics for an initial burn-in period. This will ensure that the associated magnetization chain is with high probability close to an attractor. Let $X \in \Xc$ be any spin configuration. Set the initial state $X_0=X$ and let $c_t$ be the normalized magnetization of $X_t$ at time $t$. We use $\PR_X$ and $\ER_X$ respectively to denote the underlying probability measure and associated expectation. To keep the notation light, we omit the explicit dependence on $X$ when it is clear from the context.



\begin{lemma} \label{expmagnetization}
The expected drift in $c_t$ after one step of the Glauber dynamics, starting from the configuration $X$, is given by
	\begin{equation} \label{expTstep}
		\ER(c_{t+1} - c_t \mid c_t) = \frac{1}{n} \left( \lambda(c_t) - c_t \right) + O\left(\frac{1}{n^2}\right),
	\end{equation}
where $\lambda$ is defined as in (\ref{lamb}).
\end{lemma}

\begin{proof}
From our discussion in Section \ref{mag}, we compute
\begin{align}
\ER(c_{t+1} - c_t \mid c_t) &= \frac{1}{n}(1-c_t)\frac{ p \exp\left( \alpha_1 c_t + \frac{\alpha_2}{2} c_t \left( c_t - \frac{1}{n} \right) \right) }{ p\exp\left( \alpha_1 c_t + \frac{\alpha_2}{2} c_t \left( c_t - \frac{1}{n} \right) \right)+(1-p) } \notag\\
	&\hspace{1cm}- \frac{1}{n}c_t \frac{ 1-p }{ p \exp\left( \alpha_1 \left( c_t - \frac{1}{n} \right) + \frac{\alpha_2}{2} \left( c_t - \frac{1}{n} \right)\left( c_t - \frac{2}{n} \right) \right) +(1-p)}.
\end{align}
Note that lower order fluctuations may be extracted from the exponents:
\begin{equation}
\exp\left( \alpha_1 c_t + \frac{\alpha_2}{2} c_t \left( c_t - \frac{1}{n} \right) \right)
=\exp\left( - \frac{\alpha_2}{2n}c_t \right) \exp\left( \alpha_1 c_t + \frac{\alpha_2}{2}\left(c_t\right)^2 \right),
\end{equation}
\begin{multline}
\exp\left( \alpha_1 \left( c_t - \frac{1}{n} \right) + \frac{\alpha_2}{2} \left( c_t - \frac{1}{n} \right)\left( c_t - \frac{2}{n} \right) \right)\\
=\exp\left( - \frac{\alpha_1}{n} - \frac{3 \alpha_2}{2n}c_t + \frac{\alpha_2}{n^2} \right) \exp\left( \alpha_1 c_t + \frac{\alpha_2}{2}\left(c_t\right)^2 \right),
\end{multline}
which gives
	\begin{align}
		\ER(c_{t+1} - c_t \mid c_t) &= \frac{1}{n} \left( \lambda(c_t) - c_t \right) \notag\\
			&+ \frac{\lambda(c_t)}{n}(1-c_t)\frac{ \exp\left( - \frac{\alpha_2}{2n}c_t \right) - 1 }{ 1 + \left(\frac{ p }{ 1-p }\right)\exp\left( - \frac{\alpha_2}{2n}c_t \right) \exp\left( \alpha_1 c_t + \frac{\alpha_2}{2}\left(c_t\right)^2 \right)} \notag\\
			&+ \frac{\lambda(c_t)}{n} c_t \frac{ \exp\left( - \frac{\alpha_1}{n} - \frac{3 \alpha_2}{2n}c_t + \frac{\alpha_2}{n^2} \right) - 1 }{ 1 + \left( \frac{p}{1-p} \right)\exp\left( - \frac{\alpha_1}{n} - \frac{3 \alpha_2}{2n}c_t + \frac{\alpha_2}{n^2} \right) \exp\left( \alpha_1 c_t + \frac{\alpha_2}{2}\left(c_t\right)^2 \right) }.
	\end{align}
Following standard analytical argument, the long fractional terms above are $O(1/n)$. The conclusion readily follows.
\end{proof}

Applying Lemma \ref{expmagnetization}, the following Theorem \ref{negdrift} and Corollaries \ref{cstarattraction} and \ref{cstarattraction2} show that if the associated magnetization $c_0$ of the initial configuration
is significantly different from an attractor $c^*$ but bounded away from any other solution of the fixed point equation $\lambda(c)=c$, then there is a drift of the Glauber dynamics towards a configuration whose normalized magnetization is closer to $c^*$ than the starting state, i.e. $c_t\rightarrow c^*$. Theorem \ref{negdrift} is proved when $c_0>c^*$, and an analogous result holds for $c_0<c^*$ using a similar line of reasoning. See Figure \ref{lambda(c)} for an illustration of this burn-in procedure.

\begin{theorem} \label{negdrift}
Suppose $\lambda(c^{\ast}) = c^{\ast}$ and $\lambda'(c^{\ast}) < 1$. Let $\overline{c} > c^{\ast}$ be the smallest value satisfying $\lambda(\overline{c}) = \overline{c}$. (If no such $\overline{c}$ exists, take $\overline{c} = 1$.)
Let the initial magnetization be $c_0$ with $c^{\ast} + \mu < c_0 < \overline{c} - \mu$ for some $\mu>0$. Then
there exist $\eta, d>0$ depending only on $\mu, p,\alpha_1,\alpha_2$ such that $T = dn$ and
\begin{equation}
\PR\left( c_T \le c_0 - \eta \right) \ge 1 - e^{-\Omega(\sqrt{n})}.
\end{equation}
\end{theorem}

\begin{proof}
Since $\lambda'(c^{\ast}) < 1$, $\lambda(c) - c < 0$ for $c \in [ c^{\ast} + \mu, \overline{c} - \mu ]$. By the extreme value theorem, the maximum of the smooth function $\lambda(c) - c$ is attained for some value $\underline{c}$ in the compact interval $[ c^{\ast} + \mu, \overline{c} - \mu ]$. Define $\gamma>0$ as $\gamma = -(\lambda(\underline{c}) - \underline{c})/2$. Choose $\eta > 0$ so that $[c_0 - 2\eta, c_0 + \eta] \subseteq [ c^{\ast} + \mu, \overline{c} - \mu ]$. Let
		\begin{equation} \label{Dt}
			D_t(\eta) = \left\{ c_t : c_0 - 2\eta \le c_t \le c_0 + \eta \right\}.
		\end{equation}
By Lemma \ref{expmagnetization}, for $c_t\in D_t(\eta)$ and $n$ sufficiently large, $\ER(c_{t+1} - c_t \mid c_t) \leq -\gamma/n$. Utilizing the negative drift $-\gamma/n$ of the biased random walk $c_t$ and employing a moment generating function method, we first show that ``bad'' magnetization, i.e. $c_t>c_0+\eta$ for some $0\leq t\leq T$, occurs with exponentially small probability.

Define $S_{t_1,t_2}$ as
\begin{equation}
S_{t_1,t_2} = \sum_{t = t_1 + 1}^{t_2} \left( c_t - c_{t-1} + \frac{\gamma}{2n} \right)\mathbf{1}_{ D_{t-1} }.
\end{equation}
The random variable $S_{t_1,t_2}$ records the change in ``good'' magnetization $c_t$ from $t_1$ to $t_2$, shifted by $\gamma/(2n)$ per time step. Let $\mathcal{F}_{t}$ be the natural filtration. By the tower property of expectation,
		\begin{align}
			\ER\left( e^{\theta S_{t_1,t_2}} \right)= \ER\left( e^{\theta S_{t_1,t_2 - 1}} \ER\left( e^{\theta(c_{t_2} - c_{t_2 - 1} + \frac{\gamma}{2n})\mathbf{1}_{D_{t_2 - 1}}} \mid \mathcal{F}_{t_2 - 1} \right) \right).
		\end{align}
Using linearity of expectation, we write
		\begin{align}
			&\ER\left( e^{\theta(c_{t_2} - c_{t_2 - 1} + \frac{\gamma}{2n})\mathbf{1}_{D_{t_2 - 1}}} \mid \mathcal{F}_{t_2 - 1} \right) \notag\\
				&= \sum_{k=0}^{\infty} \ER\left( \frac{ \theta^k (c_{t_2} - c_{t_2 - 1} + \frac{\gamma}{2n})^k }{ k! }\mathbf{1}_{D_{t_2 - 1}}^{k} \mid \mathcal{F}_{t_2 - 1} \right) \notag\\
				&= 1 + \ER\left( \theta(c_{t_2} - c_{t_2 - 1} + \frac{\gamma}{2n})\mathbf{1}_{D_{t_2 - 1}} \mid \mathcal{F}_{t_2 - 1} \right) + \sum_{k=2}^{\infty} \ER\left( \frac{ \theta^k (c_{t_2} - c_{t_2 - 1} + \frac{\gamma}{2n})^k }{ k! }\mathbf{1}_{D_{t_2 - 1}}^{k} \mid \mathcal{F}_{t_2 - 1} \right).
		\end{align}
	Recall that $\ER(c_{t+1} - c_t \mid c_t) \le -\gamma/n$ for $c_t\in D_t$, and so
		\begin{align}
			\ER\left( e^{\theta(c_{t_2} - c_{t_2 - 1} + \frac{\gamma}{2n})\mathbf{1}_{D_{t_2 - 1}}} \mid \mathcal{F}_{t_2 - 1} \right) &\le 1 - \frac{\gamma \theta}{2n} \mathbf{1}_{D_{t_2 - 1}} + O\left( \frac{\theta^2}{n^2}\right).
		\end{align}
Taking $\theta = c\sqrt{n}$ for a sufficiently small constant $c$, the above conditional expectation is less than $1$. Iterating this procedure gives
\begin{equation}
\ER\left(e^{\theta S_{t_1,t_2}}\right) \le \ER\left(e^{\theta S_{t_1,t_2 - 1}}\right) \le {\cdots} \le \ER\left(e^{\theta S_{t_1,t_1}}\right) = 1.
\end{equation}
By the Chernoff bound,
\begin{equation}
\PR\left(S_{t_1,t_2} \ge \frac{\eta}{2} \right) \le \frac{\ER\left( e^{\theta S_{t_1, t_2}} \right)}{e^{\theta\frac{\eta}{2}}}= e^{-\Omega(\sqrt{n})}.
\end{equation}

Consider the set
		\begin{equation} \label{Bt}
			B_{t_1,t_2}(\eta) = \left( \bigcap_{t_1 \le t < t_2} D_t(\eta) \right) \bigcap \left\{ c_{t_2} - c_{t_1} > \frac{\eta}{2} \right\},
		\end{equation}
consisting of all ``good'' magnetizations at time $t_1$ up to time $t_2$, with an increase of at least $\eta/2$ from $t_1$ to $t_2$. Subject to $c_t \in D_t$ for $t_1\leq t<t_2$ and $c_{t_2} - c_{t_1} >\eta/2$,
		\begin{align}
			S_{t_1,t_2} &= \sum_{t = t_1 + 1}^{t_2} \left( c_t - c_{t-1} + \frac{\gamma}{2n} \right) \notag\\
				&= c_{t_2} - c_{t_1} + \frac{\gamma}{2n} (t_2 - t_1) > \frac{\eta}{2},
		\end{align}
from which the containment $B_{t_1, t_2}(\eta) \subseteq \{S_{t_1,t_2} \geq \eta/2\}$ follows. Hence
\begin{equation}
\PR\left(\bigcup_{0 \le t_1 < t_2 \le T} B_{t_1, t_2}\right) \leq n^2 e^{-\Omega(\sqrt{n})}=e^{-\Omega(\sqrt{n})}.
\end{equation}
Take $n$ large enough. Suppose $c_t > c_0 + \eta$ for some $0\leq t\leq T$. Then there exists a $t_1$ such that $c_0 - 2\eta \le c_s$ for all $t_1 \le s \le t$. Define $t_2$ to be the least time greater than $t_1$ with $c_{t_2} > c_0 + \eta$. Then $c_t \in D_t$ for all $t_1 \le t<t_2$ and $c_{t_2} - c_{t_1} >\eta/2$. This implies that
\begin{equation}
\left\{ c_t : c_t > c_0 + \eta \text{ for some } 0 \le t \le T \right\} \subseteq \bigcup_{0 \le t_1 < t_2 \le T} B_{t_1, t_2},
\end{equation}
and further implies that
\begin{equation}
\label{bound}
\PR\left( c_t > c_0 + \eta \text{ for some } 0 \le t \le T \right) \leq e^{-\Omega(\sqrt{n})}.
\end{equation}

We have thus shown that the normalized magnetization $c_t$ remains below $c_0+\eta$ for all $0\leq t\leq T$ with exponentially high probability, provided that $c_0$ is suitably bounded away from any other fixed point of $\lambda$. Next we show that $c_T$ ends below $c_0-\eta$ with exponentially high probability. We prove this by showing that $c_t$ actually reaches $c_0-2\eta$ with exponentially high probability, and then by the preceding argument will have exponentially small probability of increasing to $c_0-\eta$. We have
\begin{equation}
\PR\left(c_t \ge c_0 - 2\eta \text{ for all } 0 \le t \le T\right) \le \PR\left( \bigcap_{0 \le t \le T} D_t(\eta) \right) + \PR\left( c_t > c_0 + \eta \text{ for some } 0 \le t \le T \right).
\end{equation}
Subject to $c_t \in D_t(\eta)$ for $0\leq t\leq T$ and noticing that at worst $c_0=1$ and $c_T=0$,
\begin{align}
			S_{0,T} &= \sum_{t = 1}^{T} \left( c_t - c_{t-1} + \frac{\gamma}{2n} \right) \mathbf{1}_{ D_{t-1} } \notag\\
				&= c_T - c_0 + \frac{\gamma}{2n}T \ge -1 + \frac{\gamma}{2n} T.
\end{align}
Using the Chernoff bound on $S_{0, T}$ and assume that $d>2/\gamma$,
		\begin{multline}
			\PR\left(c_t \ge c_0 - 2\eta \text{ for all } 0 \le t \le T\right) \le \PR\left( S_{0,T} \ge -1 + \frac{\gamma}{2n} T \right) + e^{-\Omega(\sqrt{n})} \\
				\le \frac{\ER\left( e^{\theta S_{0,T}} \right)}{e^{\theta \left( -1 + \frac{\gamma}{2n} T \right)}} + e^{-\Omega(\sqrt{n})} = e^{-\Omega(\sqrt{n})}.
		\end{multline}
Finally,
		\begin{align}
			\PR\left( c_T \ge c_0 - \eta \right)  &\le \PR\left( c_T \ge c_0 - \eta \text{ and } c_t < c_0 - 2\eta \text{ for some } 0 \le t \le T \right) \notag\\
				&+ \PR\left( \left\{ c_t \ge c_0 - 2\eta \text{ for all } 0 \le t \le T \right\} \right) \le e^{-\Omega(\sqrt{n})},
		\end{align}
where to bound the first probability on the right, we apply the bound on ``bad'' magnetization (\ref{bound}) with minor adaptation: initial magnetization $c_t$ (for some $0\leq t\leq T$) in place of $c_0$, time interval under consideration $[t, T]$ in place of $[0, T]$, and the increase in magnetization $c_0-2\eta \rightarrow c_0-\eta$ in place of $c_0 \rightarrow c_0+\eta$.
\end{proof}

Repeated application of Theorem \ref{negdrift} shows that after a burn-in period on the order of $O(n)$, any suitably chosen configuration ends up close to an attractor $c^{\ast}$ with exponentially high probability. Recall the definition of high temperature phase and low temperature phase from Section \ref{phase}. The following corollaries are immediate.

\begin{corollary} \label{cstarattraction}
In the high temperature phase, suppose that $c^{\ast}$ is the unique solution to $\lambda(c)=c$ and $\lambda'(c^{\ast}) < 1$. For any $\varepsilon > 0$, there exists $\alpha > 0$ such that for any initial configuration with associated magnetization $c_0$, when $t \ge \alpha n$ we have
\begin{equation}
\PR\left( c_t \ge c^{\ast} + \varepsilon\right) \le e^{-\Omega(\sqrt{n})}
\end{equation}
and
\begin{equation}
\PR\left( c_t \le c^{\ast} - \varepsilon\right) \le e^{-\Omega(\sqrt{n})}.
\end{equation}
\end{corollary}

\begin{corollary} \label{cstarattraction2}
In the low temperature phase, suppose that $c^{\ast}$ is a solution to $\lambda(c)=c$ and $\lambda'(c^{\ast}) < 1$. Take $\varepsilon > 0$. If the associated magnetization $c_0$ for some initial configuration satisfies $c^{\ast} - \varepsilon \le c_0 \le c^{\ast} + \varepsilon$, then there exists $\beta > 0$,
\begin{equation}
\PR\left( \sup_{0 < t < e^{\beta \sqrt{n}}} c_t \ge c^{\ast} + 2 \varepsilon \right) \le e^{-\Omega(\sqrt{n})}
\end{equation}
and
\begin{equation}
\PR\left( \inf_{0 < t < e^{\beta \sqrt{n}}} c_t \le c^{\ast} - 2 \varepsilon \right) \le e^{-\Omega(\sqrt{n})}.
\end{equation}
\end{corollary}

\section{Fast mixing at high temperature}
\label{fast}
In this section we study the mixing time of the Glauber dynamics in the high temperature phase. We first establish an upper bound $O(n\log n)$ using path coupling techniques of Bubley and Dyer \cite{BD}. Consider two arbitrary spin configurations $X, Y \in \Xc$. Taking ``attractive'' parameters $\alpha_i\geq 0$ ensures that we may apply a monotone coupling on the chain $(X_t, Y_t)$: $X_t$ is a version of the Glauber dynamics with starting state $X$ and $Y_t$ is a version of the Glauber dynamics with starting state $Y$, if $X_0 \leq Y_0$ then $X_t \leq Y_t$ for all $t$. We write $\PR_{X,Y}$ and $\ER_{X,Y}$ for the underlying probability measure and associated expectation. To keep the notation light, we omit the explicit dependence on $X$ and $Y$ when it is clear from the context. To understand how far apart $X_t$ and $Y_t$ are, we introduce Hamming distance, which records the number of vertices where the two configurations disagree. Define $\rho : \Xc \times \Xc \to \{0,\dots,n\}$ by
\begin{equation}
\rho(X,Y) = \sum_{i=1}^{n} \left| X(i) - Y(i) \right|.
\end{equation}
Following standard contraction argument, it suffices to estimate the average distance after one Glauber update between two coupled configurations $X$ and $Y$ with Hamming distance $\rho(X, Y)=1$.

\begin{lemma} \label{couplinglemma}
Assume that $\sup_{0 \le c \le 1} \lambda'(c) < 1$. Let $X, Y \in \Xc$ be two spin configurations satisfying $X \le Y$ and $\rho(X, Y)=1$. Set the initial state $X_0=X$ and $Y_0=Y$. Then there exists $\delta > 0$ depending only on $p, \alpha_1, \alpha_2$ such that a single step of the Glauber dynamics can be coupled when $n$ is sufficiently large:
\begin{equation}
\ER\left( \rho(X_1,Y_1) \right) \le e^{-\delta /n}.
\end{equation}
\end{lemma}

\begin{proof}
Let $X, Y \in \Xc$ be two configurations such that $X \le Y$ and there exists a single vertex $i$ such that $X(i) = 0$ and $Y(i) = 1$. Let $U$ be a uniform random variable on $[0,1]$. We apply the standard monotone coupling, where $U$ is used as the common source of noise to update both chains so that they agree as often as possible. From the mechanism described in Section \ref{model}, the chain evolves by selecting a vertex $j$ uniformly at random and updating the spin at $j$. Set
\begin{equation} \label{Xw}
			X_1(j) = \left\{ \begin{array}{ll}
						1 & U \le p_{+}(X, j), \\
						0 & U > p_{+}(X, j),
					\end{array} \right.
\hspace{1cm}
			Y_1(j) = \left\{ \begin{array}{ll}
						1 & U \le p_{+}(Y, j), \\
						0 & U > p_{+}(Y, j),
					\end{array} \right.
\end{equation}
and $X_1(k) = X(k)$ and $Y_1(k) = Y(k)$ for all $k \neq j$. Define the function $f(S)$ as
\begin{equation}
f(S) = \frac{ p \exp\left( \frac{\alpha_1}{n} S + \frac{\alpha_2}{2n^2}S(S-1) \right) }{ p \exp\left( \frac{\alpha_1}{n} S + \frac{\alpha_2}{2n^2}S(S-1) \right) + (1-p) }.
\end{equation}
If $j = i$, then $p_{+}(X, j)=f\left(S(X, j)\right)=p_{+}(Y, j)$ and so $\rho(X_1,Y_1) = 0$. For $j \neq i$, we have $p_{+}(X, j)=f\left(S(X, j)\right)$ while $p_{+}(Y, j)=f\left(S(X, j)+1\right)$, where $0\leq S(X, j)\leq n-2$. Since $f$ is a smooth and increasing function, this shows that $p_{+}(X, j)\leq p_{+}(Y, j)$. Hence $\rho(X_1,Y_1) = 2$ if $p_{+}(X, j) < U \le p_{+}(Y, j)$ and $\rho(X_1, Y_1)=1$ otherwise.

We wish to find an upper bound for
		\begin{equation} \label{conexp}
			\ER \left( \rho(X_1,Y_1) \right) = 1 - \frac{1}{n} + \frac{1}{n} \sum_{j \neq i} \left(p_{+}(Y, j) - p_{+}(X, j)\right).
		\end{equation}
To that end, we compute, by the mean value theorem
\begin{equation}
p_{+}(Y, j) - p_{+}(X, j)=f\left(S(X, j)+1\right)-f\left(S(X, j)\right)=\frac{g'(\overline{c})}{n},
\end{equation}
where $0\leq \overline{c}\leq 1$ and $g'(c)$ is defined as
\begin{equation}
g'(c)=\frac{p(1-p)\exp\left(\alpha_1 c+\frac{\alpha_2}{2}c\left(c-\frac{1}{n}\right)\right)}{p\exp\left(\alpha_1 c+\frac{\alpha_2}{2}c\left(c-\frac{1}{n}\right)\right)+\left(1-p\right)}\left(\alpha_1+\alpha_2 c-\frac{\alpha_2}{2n}\right).
\end{equation}
Compare $g'(c)$ against $\lambda'(c)$, where $\lambda(c)$ is defined as in (\ref{lamb}),
\begin{equation}
\lambda'(c)=\frac{p(1-p)\exp\left(\alpha_1 c+\frac{\alpha_2}{2}c^2\right)}{p\exp\left(\alpha_1 c+\frac{\alpha_2}{2}c^2\right)+\left(1-p\right)}\left(\alpha_1+\alpha_2 c\right).
\end{equation}
Via standard analytical argument as in the proof of Lemma \ref{expmagnetization}, the difference $g'(c)-\lambda'(c)=O(1/n)$. Therefore
		\begin{align}
\label{sup}
			\ER\left( \rho(X_1,Y_1) \right) &\le 1 - \frac{1}{n} + \frac{1}{n} \sum_{j \neq i} \left(\frac{1}{n}\sup_{0 \le c \le 1}\lambda'(c)  +O\left(\frac{1}{n^2}\right)\right) \notag\\
				&\le 1 - \frac{1}{n} + \frac{n-1}{n^2}\sup_{0 \le c \le 1} \lambda'(c)+O\left(\frac{1}{n^2}\right)  \notag\\
				&\le 1 - \frac{1 - \sup_{0 \le c \le 1} \lambda'(c)}{n}+O\left(\frac{1}{n^2}\right).
		\end{align}
Define $\delta>0$ as $\delta= \left(1 - \sup_{0 \le c \le 1} \lambda'(c)\right)/2$. Then for $n$ sufficiently large, we obtain
\begin{equation}
\ER\left( \rho(X_1,Y_1) \right) \le 1 - \frac{\delta}{n} \le e^{-\delta/n}.
\end{equation}
\end{proof}

The requirement $\sup_{0 \le c \le 1} \lambda'(c) < 1$ in Lemma \ref{couplinglemma} may be weakened. By Corollary \ref{cstarattraction}, in the high temperature phase, with exponentially high probability, after $O(n)$ time steps the associated magnetization of all configurations are within an $\varepsilon$-neighborhood of the unique solution $c^*$ of $\lambda$ with $\lambda'(c^*)<1$. The supremum referenced in (\ref{sup}) thus need not be taken over the entire interval $[0, 1]$ but just $[c^*-\varepsilon, c^*+\varepsilon]$, and is guaranteed to be less than $1$ using smoothness of $\lambda$. Now take any two configurations $X, Y \in \Xc$ with $\rho(X, Y)=k$, where $1\leq k\leq n$. (At worst $X(i)=0$ and $Y(i)=1$ for all $i\in \{1, \ldots, n\}$.) There is a sequence of states $X_0, \ldots, X_k$ such that $X_0=X$, $X_k=Y$, and each neighboring pair $X_i, X_{i+1}$ are unit Hamming distance apart. Applying Lemma \ref{couplinglemma} for configurations at unit distance, we have $\ER\left(\rho(X_1,Y_1)\right)\leq ne^{-\delta/n}$.
Iterating gives
\begin{equation}
\label{contraction}
\ER\left(\rho(X_t,Y_t)\right)\leq ne^{-\delta t/n}.
\end{equation}

\begin{theorem}
\label{match}
In the high temperature phase, the mixing time for the Glauber dynamics is $O(n\log n)$.
\end{theorem}

\begin{proof}
By Theorem 14.6 and Corollary 14.7 of \cite{MCMT}, (\ref{contraction}) implies
		\begin{equation} \label{tmixup}
			t_{\text{mix}}(\varepsilon) \le \left\lceil \frac{n\left( \log n - \log \varepsilon \right)}{\delta}  \right\rceil.
		\end{equation}
Setting $\varepsilon=1/4$,
\begin{equation}
t_{\text{mix}} \le \left\lceil \frac{n\left( \log n + \log 4 \right)}{\delta}  \right\rceil.
\end{equation}
\end{proof}

Next in Theorem \ref{matching}, by checking the total variance distance from the stationary distribution at time $t^*:=(n\log n)/4$, we establish a matching lower bound $\Omega(n \log n)$ for the Glauber dynamics in the high temperature phase. Together with Theorem \ref{match}, the correct order for the mixing time, $\Theta(n\log n)$, is validated.

\begin{theorem}
\label{matching}
In the high temperature phase, the mixing time for the Glauber dynamics is $\Omega(n\log n)$.
\end{theorem}

\begin{proof}
Let $f(X)$ count the number of vertices with spin $1$ in a configuration $X \in \Xc$. Denote by $\gamma$ the spectral gap associated with the Glauber dynamics. By Lemma 13.12 and Remark 13.13 of \cite{MCMT},
		\begin{equation} \label{gapinf}
			\gamma \leq \frac{\mathbb{\varepsilon}(f)}{\mathbb{V}ar_{\pi}(f)},
		\end{equation}
	where the Dirichlet form
\begin{equation}
\mathbb{\varepsilon}(f) = \frac{1}{2} \sum_{X,Y \in \Xc} \left( f(X) - f(Y) \right)^2 \pi(X) P(X,Y),
\end{equation}
and the variance under stationary distribution
\begin{equation}
\mathbb{V}ar_{\pi}(f)=\sum_{X\in \Xc} \left(f(X)\right)^2 \pi(X)-\left(\sum_{X\in \Xc} f(X)\pi(X)\right)^2.
\end{equation}
From the mechanism described in Section \ref{model}, for configurations $X, Y\in \Xc$, $P(X, Y)$ (\ref{transmat}) is zero unless $X$ and $Y$ differ at at most one vertex. This implies that
\begin{equation}
\mathbb{\varepsilon}(f)\leq \frac{1}{2} \sum_{X,Y \in \Xc} \pi(X)P(X, Y)=\frac{1}{2} \sum_{X\in \Xc}\pi(X)=\frac{1}{2},
\end{equation}
and when applied to (\ref{gapinf}), further implies that $2\mathbb{V}ar_\pi(f)\leq 1/\gamma$. Hence
\begin{equation}
\log 2 \left( 2\mathbb{V}ar_{\pi}(f) - 1 \right) \le \log 2 \left( \frac{1}{\gamma} - 1 \right) \le t_{\text{mix}},
\end{equation}
where the second inequality uses spectral representation techniques (for details, see for example Theorem 12.4 of \cite{MCMT}). By Theorem \ref{match}, $t_{\text{mix}} = O(n\log n)$, which then gives $\mathbb{V}ar_{\pi}(f) = O(n \log n)$. Let configuration $X$ be chosen according to the stationary distribution $\pi$. By Chebyshev's inequality,
\begin{equation}
\pi\left( \left| f(X) - \ER_{\pi}(f(X)) \right| > n^{2/3} \right) \le \frac{\mathbb{V}ar_{\pi}(f(X))}{n^{4/3}} = O(n^{-1/3} \log n).
\end{equation}
Therefore asymptotically
\begin{equation}
\label{bd1}
\pi\left( \left| f(X) - \ER_{\pi}(f(X)) \right| \le n^{2/3} \right) \rightarrow 1.
\end{equation}

Let $X^{+}, X^- \in \Xc$ be configurations such that $X^{+}(i)=1$ and $X^-(i)=0$ for every $i\in \{1, \ldots, n\}$. Assume that $X^+$ and $X^-$ are coupled using the standard monotone coupling as described in the proof of Lemma \ref{couplinglemma}. So $X^+_t \ge X^-_t$ for all $t$. Let $R_{t^{\ast}}$ denote the number of vertices not yet selected by the Glauber Dynamics by time $t^{\ast}$. By a coupon collecting argument,
\begin{equation}
\ER(R_{t^*}) \asymp n^{3/4} \text{ and }\mathbb{V}ar(R_{t^*}) \leq \ER(R_{t^*}).
\end{equation}
(For details, see for example Lemma 7.12 of \cite{MCMT}.) Let $\varepsilon > 0$. By Chebyshev's inequality, asymptotically
		\begin{equation}
			\PR\left( \left| R_{t^{\ast}} - \ER\left( R_{t^{\ast}} \right) \right| > (1-\varepsilon) \ER\left( R_{t^{\ast}} \right) \right) \le \frac{ \mathbb{V}ar\left( R_{t^{\ast}} \right) }{ (1-\varepsilon)^2 \left(\ER\left( R_{t^{\ast}} \right)\right)^2  }\leq O(n^{-3/4}),
		\end{equation}
which using set containment implies that
\begin{equation}
\PR\left( R_{t^{\ast}}< \varepsilon\ER\left(R_{t^{\ast}}\right)\right) \le \PR\left(\left| R_{t^{\ast}} - \ER\left( R_{t^{\ast}} \right) \right| > (1-\varepsilon) \ER\left( R_{t^{\ast}} \right) \right) \leq O(n^{-3/4}).
\end{equation}
It follows that $R_{t^{\ast}} = \Omega(n^{3/4})$ with probability tending to $1$ asymptotically. Since $f(X^{+}_{t^{\ast}}) - f(X^{-}_{t^{\ast}})\geq R_{t^*}$, we conclude that $f(X^{+}_{t^{\ast}}) - f(X^{-}_{t^{\ast}})= \Omega(n^{3/4})$ with probability tending to $1$ asymptotically.

Define sets $A$ and $B$ respectively as
\begin{equation}
A:=\left\{\left| f(X^{+}_{t^{\ast}}) - \ER_{\pi}(f(X)) \right| \leq n^{2/3}\right\}
\end{equation}
and
\begin{equation}
B:=\left\{\left| f(X^{-}_{t^{\ast}}) - \ER_{\pi}(f(X)) \right| \leq n^{2/3}\right\}.
\end{equation}
Then by the triangle inequality, their intersection $A\cap B$, if nonempty, satisfies $f(X^{+}_{t^{\ast}}) - f(X^{-}_{t^{\ast}}) = O(n^{2/3})$. Since $n^{3/4}>n^{2/3}$, this contradicts with what was established in the previous paragraph. Therefore the sets $A$ and $B$ are asymptotically disjoint and so one of them has probability bounded above by $1/2 - o(1)$. Without loss of generality, say
\begin{equation}
\label{bd2}
\PR\left( \left| f(X^{+}_{t^{\ast}}) - \ER_{\pi}(f(X)) \right| \leq n^{2/3} \right) \le \frac{1}{2} - o(1).
\end{equation}
By definition of the total variation distance and the bounds (\ref{bd1}) (\ref{bd2}),
\begin{equation}
\norm{P^{t^{\ast}}(X^{+},{\cdot}) - \pi}_{\text{TV}} \geq \left| P^{t^{\ast}}(X^{+}, A) - \pi(A) \right| \ge 1 - \frac{1}{2} + o(1).
\end{equation}
Since at $t_{\text{mix}}$, the default for $t_{\text{mix}}(\varepsilon)$, the distance must be less than or equal to $1/4$, this shows that the mixing time is asymptotically bigger than $t^{\ast} = (n \log n)/4$ proving the lower bound.
\end{proof}

\section{Slow mixing at low temperature}
\label{slow}
In this section we study the mixing time of the Glauber dynamics in the low temperature phase. Rather than analyzing the spin update probability $\lambda$ (\ref{lamb}) directly, we find asymptotic expressions for components of the partition function $Z$ (see Definition \ref{def}). For $k\in \{0, \ldots, n\}$, define $A_k = \{ X : \left| \{ i : X(i) = 1 \} \right| = k \}$. That is, the set $A_k$ consists of spin configurations $X\in \Xc$ whose number of vertices with spin $1$ is $k$ and number of vertices with spin $0$ is $n - k$. Then
	\begin{equation} \label{piak}
		\pi(A_k) = \frac{1}{Z} {n \choose k} \exp\left( \frac{\alpha_1}{n} {k \choose 2} + \frac{\alpha_2}{n^2} {k \choose 3} \right) p^k (1-p)^{n-k}.
	\end{equation}
Let $a_k$ be such that $a_k = Z \pi(A_k)$. Notice that $\sum_{k=0}^n a_k=Z$.

\begin{lemma} \label{logacnlemma}
	Let $c\in [0, 1]$ and $a_k$ be defined as above, we have
		\begin{equation} \label{logacn}
			\log(a_{\lfloor cn \rfloor}) = n (\varphi_{p, \alpha_1, \alpha_2}(c) + o(1)),
		\end{equation}
	where
		\begin{equation} \label{phifunc}
			\varphi(c) = \frac{\alpha_1}{2}c^2 + \frac{\alpha_2}{6}c^3 - c\log \frac{c}{p} - (1-c) \log \frac{1-c}{1-p}.
		\end{equation}
\end{lemma}

\begin{proof}
Stirling's formula states that
\begin{equation}
n! \asymp \sqrt{2\pi} e^{-n} n^n \sqrt{n}.
\end{equation}
Setting $k = \lfloor cn \rfloor$, the binomial coefficient admits an asymptotic formula:
\begin{equation}
			{n \choose k} \asymp \frac{1}{\sqrt{2\pi}c^{cn}\sqrt{c}(1-c)^{(1-c)n}\sqrt{1-c}\sqrt{n}}.
\end{equation}
Hence
		\begin{align}
			\log(a_{\lfloor cn \rfloor}) & \asymp n\left( \frac{\alpha_1}{2}c^2  + \frac{\alpha_2}{6} c^3  + c\log{p} + (1-c)\log(1-p) - c\log{c} - (1-c)\log(1-c) \right) \notag \\
				&\text{\hspace{1cm}}- \log{\left( \sqrt{2\pi c(1-c) n} \right)} \notag \\
				&= n (\varphi(c) + o(1)).
		\end{align}
\end{proof}

Next in Lemmas \ref{lem1} and \ref{lem2}, we reveal a deep relationship between $\varphi$ (whose derivative is commonly referred to as the ``free energy density'') and the spin update probability $\lambda$. As we will see, local maximizers for $\varphi$ correspond to fixed points of $\lambda$, and concavity of $\varphi$ at the local maximizer (indicating whether it is a local maximum or minimum) translates to the attractor/repellor characterization on the fixed point of $\lambda$ previously described in Section \ref{phase}.

\begin{lemma}
\label{lem1}
Let $\lambda$ and $\varphi$ be respectively defined as in (\ref{lamb}) and (\ref{phifunc}). Then
$\lambda(c) = c  \iff  \varphi'(c) = 0$.
\end{lemma}

\begin{proof}
We have the following string of equivalences
	\begin{align}
		\lambda(c)=c &\iff \left( \frac{p}{1-p} \right)\exp\left( \alpha_1 c + \frac{\alpha_2}{2} c^2 \right) = \frac{c}{1-c} \notag \\
			&\iff \alpha_1 c + \frac{\alpha_2}{2} c^2 + \log \frac{p}{1-p} - \log \frac{c}{1-c} = 0 \notag \\
			&\iff \varphi'(c) = 0.
	\end{align}
\end{proof}

\begin{lemma}
\label{lem2}
Let $\lambda$ and $\varphi$ be respectively defined as in (\ref{lamb}) and (\ref{phifunc}). Suppose $\varphi'(c) = 0$. Then $\varphi''(c) > 0 \iff \lambda'(c) > 1$, $\varphi''(c) < 0 \iff \lambda'(c) < 1$, and $\varphi''(c) = 0 \iff \lambda'(c) = 1$.
\end{lemma}

\begin{proof}
Suppose $\varphi'(c) = 0$. From Lemma \ref{lem1}, $\lambda(c)=c$. Therefore
\begin{equation}
		\lambda'(c) = c(1-c) \left( \alpha_1 + \alpha_2 c \right).
\end{equation}
The claim readily follows since
\begin{equation}
\varphi''(c) = \alpha_1 + \alpha_2 c - \frac{1}{c(1-c)}.
\end{equation}
\end{proof}

Using a conductance argument whose idea goes back at least to Griffiths et al. \cite{G}, we now show that in the region where $\varphi(c)$ has at least two local maximizers, the mixing time for the Glauber dynamics is at least exponential. In the language of $\lambda$, this establishes exponentially slow mixing of the Glauber dynamics when $\lambda(c)=c$ has at least two solutions $c$ satisfying $\lambda'(c)<1$. Define the bottleneck ratio (Cheeger constant) of a Markov chain with stationary distribution $\pi$ as
\begin{equation}
\Phi_{\ast} = \min_{S: \pi(S) \le 1/2} \frac{Q(S,S^c)}{\pi(S)},
\end{equation}
where $S$ is a set in the configuration space and $Q$ is the edge measure given by
\begin{equation}
Q(X,Y) = \pi(X)P(X,Y) \text{ and } Q(A,B) = \sum_{X \in A, Y \in B} Q(X,Y).
\end{equation}
Recall that $P(X, Y)$ (\ref{transmat}) is the transition probability from configuration $X$ to configuration $Y$, and so $Q(A, B)$ is the probability of moving from set $A$ to set $B$ in one step of the chain when starting from the stationary distribution.

\begin{theorem}
\label{slowmix}
In the low temperature phase, the mixing time for the Glauber dynamics is $e^{\Omega(n)}$.
\end{theorem}

\begin{proof}
Notice that $\varphi'(c) \rightarrow \infty$ as $c \rightarrow 0$ and $\varphi'(c) \rightarrow -\infty$ as $c \rightarrow 1$, so the local maximizers of $\varphi$ are contained in $(0,1)$. Let $c_1$ be the smallest and $c_2$ be the largest local maximizer of $\varphi$. We have $c_1 \neq c_2$. There exists $\varepsilon > 0$ such that for all $c < c_1$ and $c_1 < c \le c_1 + \varepsilon$, $\varphi(c) < \varphi(c_1)$, while for all $c_2 - \varepsilon \le c < c_2$ and $c > c_2$, $\varphi(c) < \varphi(c_2)$, with $c_1 + \varepsilon < c_2 - \varepsilon$. Define the following two sets
		\begin{equation} \label{setS1}
			S_1 = \{ A_0, \dots, A_{\lfloor (c_1 + \varepsilon)n \rfloor} \}
		\end{equation}
	and
		\begin{equation} \label{setS2}
			S_2 = \{ A_{\lfloor (c_2 - \varepsilon)n \rfloor}, \dots, A_n \}.
		\end{equation}
For $n$ large enough, $S_1 \cap S_2 = \varnothing$, $A_{\lfloor (c_1 + \varepsilon)n \rfloor} \neq A_{\lfloor c_1 n \rfloor}$, and $A_{\lfloor (c_2 - \varepsilon)n \rfloor} \neq A_{\lfloor c_2 n \rfloor}$. Since $S_1$ and $S_2$ are disjoint, at least one of them has probability bounded above by $1/2$. Without loss of generality, suppose $S_1$ is such that $\pi(S_1) \le \frac{1}{2}$. Then
\begin{equation}
\label{transprob}
\Phi_* \leq \frac{Q(S_1, S_1^{c})}{\pi(S_1)} = \frac{1}{\pi(S_1)} {\sum_{i=0}^{\lfloor (c_1 + \varepsilon)n \rfloor} \sum_{X\in A_i}\pi(X) \sum_{j= \lfloor (c_1 + \varepsilon)n \rfloor+1}^{n}\sum_{Y\in A_j} P(X, Y)}.
\end{equation}
Since a single step of the Glauber dynamics only changes the value of the normalized magnetization by at most $1/n$, the only non-zero transition probability in (\ref{transprob}) is the transition from $A_{\lfloor (c_1 + \varepsilon)n \rfloor}$ to $A_{\lfloor (c_1 + \varepsilon)n \rfloor}+1$. It follows that
\begin{align}
\Phi_{*} &\le \frac{1}{\pi(S_1)}{\sum_{i=0}^{\lfloor (c_1 + \varepsilon)n \rfloor} \sum_{X\in A_i}\pi(X)} \le \frac{\pi(A_{\lfloor (c_1 + \varepsilon)n \rfloor})}{\pi(A_{\lfloor c_1 n \rfloor})}\notag\\ &= \frac{a_{\lfloor (c_1 + \varepsilon)n \rfloor}}{a_{\lfloor c_1 n \rfloor}} \asymp \frac{e^{n \varphi(c_1 + \varepsilon)}}{e^{n \varphi(c_1)}} = e^{n (\varphi(c_1 + \varepsilon) - \varphi(c_1))},
\end{align}
where the asymptotics are derived in Lemma \ref{logacnlemma}. Define $\delta > 0$ as $\delta = \left(\varphi(c_1) - \varphi(c_1 + \varepsilon)\right)/2$. Then the bottleneck ratio satisfies $\Phi(S_1) \le e^{-\delta n}$. Using Theorem 7.3 of \cite{MCMT},
\begin{equation}
t_{\text{mix}} \ge \frac{1}{4\Phi_{\ast}} \ge \frac{e^{\delta n}}{4}=e^{\Omega(n)}.
\end{equation}
\end{proof}

Call a Markov chain local if at most $o(n)$ vertices are selected in each step. The argument used in the proof of Theorem \ref{slowmix} actually shows that in the low temperature phase, the mixing is exponentially slow for any local Markov chain, with the (single-site) Glauber dynamics being one such instance. We remark that there is a difference in the qualitative nature of the phase transition investigated in this paper as compared with that in the standard statistical physics literature. While the asymptotic phase transitions in the rigorous statistical physics sense occur at parameter values giving non-unique global maximizers of the free energy density, the asymptotic transition from high temperature phase to low temperature phase arises as a consequence of the non-uniqueness of local maximizers for the free energy density. This discrepancy may not come as a surprise, since in simulations it is often hard to distinguish between a local maximizer and a global maximizer and the algorithm may become trapped at a local maximizer; one solution might be to add controlled moves based on network geometry.

\section{Slower burn-in along critical curve}
\label{critical}
In this section we study the mixing time of the Glauber dynamics along the critical curve, corresponding to parameters for which $\lambda(c)=c$ admits a unique solution $c$ with $\lambda'(c)=1$. We first identify explicitly the high temperature vs. low temperature phase. As explained in Section \ref{slow}, the two phases may be alternatively determined by whether there is a unique local maximizer for $\varphi$ (\ref{phifunc}). The phase identification thus reduces to a $3$-dimensional intricate calculus problem. Though straight-forward as it sounds, various tricks are needed to solve it analytically. The crucial idea is to minimize the effect of the parameters $p, \alpha_1, \alpha_2$ on the free energy density $\varphi$ one by one. (See \cite{Yin2013} for more details of the calculation in a related model.) Denote by $l(c):=\varphi(c)-\log(1-p)$.

\begin{proposition}
\label{max}
	Fix $\alpha_2$. Consider the maximization problem for
\begin{equation}
l_{\alpha_2}(c; p,\alpha_1) = \left(\log\frac{p}{1-p}\right)c + \frac{\alpha_1}{2}c^2 + \frac{\alpha_2}{6}c^3 - c\log c - (1-c)\log(1-c)
\end{equation}
on the interval $[0,1]$, where $0 < p < 1$ and $-\infty < \alpha_1 < \infty$ are parameters. Then there is a V-shaped region in the $(p,\alpha_1)$-plane with corner point $(p^c,\alpha_1^c)$,
\begin{equation}
\label{p}
		p^c=\frac{\bar{c}\exp\bigg(\frac{4\bar{c}-3}{2(1-\bar{c})^2}\bigg)}{\bar{c}\exp\bigg(\frac{4\bar{c}-3}{2(1-\bar{c})^2}\bigg)+(1-\bar{c})},
\end{equation}
\begin{equation}
\label{alpha1}
			\alpha_1^c =  \frac{2-3\bar{c}}{\bar{c}(1-\bar{c})^2},
\end{equation}
		where $\bar{c}$ is uniquely determined by
\begin{equation}
\label{alpha2}
			\alpha_2 = \frac{2\bar{c}-1}{\bar{c}^2(1-\bar{c})^2}.
\end{equation}
Outside this region, $l_{\alpha_2}(c)$ has only one local maximizer $c^*$. Inside this region, $l_{\alpha_2}(c)$ has exactly two local maximizers $c_1^*$ and $c_2^*$.
\end{proposition}

\begin{proof}
The location of maximizers of $l_{\alpha_2}(c)$ on the interval $[0, 1]$ is closely related to the properties of its derivatives:
\begin{equation*}
			l_{\alpha_2}'(c) = \log\frac{p}{1-p} + \alpha_1 c + \frac{\alpha_2}{2}c^2 - \log\frac{c}{1-c},	
\end{equation*}
\begin{equation*}
			l_{\alpha_2}''(c) = \alpha_1 + \alpha_2 c - \frac{1}{c(1-c)},
\end{equation*}	
\begin{equation}
			l_{\alpha_2}'''(c) = \alpha_2 + \frac{1-2c}{c^2(1-c)^2}.
\end{equation}	
We check that $l_{\alpha_2}'''(c)$ is monotonically decreasing on $[0, 1]$, $l'''_{\alpha_2}(0) = \infty$, and $l'''_{\alpha_2}(1) = -\infty$. Thus there is a unique $\bar{c}$ in $(0, 1)$ such that $l_{\alpha_2}'''(\bar{c}) = 0$, with $l_{\alpha_2}'''(c) > 0$ for $c < \bar{c}$ and $l_{\alpha_2}'''(c) < 0$ for $c > \bar{c}$. Since the correspondence between $\alpha_2$ and $\bar{c}$ is one-to-one, we may describe $\alpha_2$ by (\ref{alpha2}).

This implies that $l''_{\alpha_2}(c)$ is increasing from $0$ to $\bar{c}$, and decreasing from $\bar{c}$ to $1$, with the global maximum achieved at $\bar{c}$, where
\begin{equation}
			l_{\alpha_2}''(\bar{c}) = \alpha_1 + \frac{3\bar{c}-2}{\bar{c}(1-\bar{c})^2}.
\end{equation}
Let $\alpha_1^c$ be defined as in (\ref{alpha1}) so that $l_{\alpha_2}''(\bar{c}; \alpha_1^c) = 0$. It follows that for $\alpha_1 \le \alpha_1^c$, $l_{\alpha_2}''(c) \le 0$ on the entire interval $[0,1]$; whereas for $\alpha_1 > \alpha_1^c, l_{\alpha_2}''(c)$ takes on both positive and negative values, and we denote the transition points by $c_1$ and $c_2$ ($c_1 < \bar{c} < c_2$). For fixed $\alpha_2$, $c_1$ and $c_2$ are solely determined by $\alpha_1$, and vice versa. Let $m(c)= \alpha_1 - l_{\alpha_2}''(c)$ so that $\alpha_1 = m(c_1) = m(c_2)$. We have $m(0) = m(1) = \infty, m(c)$ is decreasing from $0$ to $\bar{c}$, and increasing from $\bar{c}$ to $1$.

We proceed to analyze properties of $l_{\alpha_2}'(c)$ and $l_{\alpha_2}(c)$ on the interval $[0, 1]$. For $\alpha_1 \le \alpha_1^c, l_{\alpha_2}'(c)$ is monotonically decreasing.  		For $\alpha_1 > \alpha_1^c, l_{\alpha_2}'(c)$   is decreasing from $0$ to $c_1$, increasing from $c_1$ to $c_2$, then decreasing again from $c_2$ to $1$. We write down the explicit expressions of $l_{\alpha_2}'(c_1)$ and $l_{\alpha_2}'(c_2)$:
\begin{equation*}
			l_{\alpha_2}'(c_1) = \log\frac{p}{1-p} + \frac{1}{1-c_1} - \log\frac{c_1}{1-c_1} + \frac{1-2\bar{c}}{2\bar{c}^2(1-\bar{c})^2}c_1^2,
\end{equation*}
\begin{equation}
			l_{\alpha_2}'(c_2) = \log\frac{p}{1-p} + \frac{1}{1-c_2} - \log\frac{c_2}{1-c_2} + \frac{1-2\bar{c}}{2\bar{c}^2(1-\bar{c})^2}c_2^2.
\end{equation}
Notice that $l_{\alpha_2}(c)$ is a bounded continuous function, $l_{\alpha_2}'(0) = \infty $, and $l_{\alpha_2}'(1) = -\infty$, so $l_{\alpha_2}(c)$ cannot be maximized at $0$ or $1$. For $\alpha_1 \le \alpha_1^c, l_{\alpha_2}'(c)$ crosses the c-axis 			only once, going from positive to negative. Thus $l_{\alpha_2}(c)$ has a unique local maximizer $c^*=\bar{c}$. For $\alpha_1 > \alpha_1^c$, the situation is more complicated. 		If $l_{\alpha_2}'(c_1) \ge 0$ (resp. $l_{\alpha_2}'(c_2) \le 0$), $l_{\alpha_2}(c)$ has a unique local maximizer at a point $c^* > c_2$ (resp. $c^* < c_1)$. If 					$l_{\alpha_2}'(c_1) < 0 < l_{\alpha_1}'(c_2)$,  then $l_{\alpha_2}(c)$ has two local maximizers $c_1^*$ and $c_2^*$, with $c_1^* < c_1 < \bar{c} < c_2 < c_2^*$.

Let
\begin{equation}
			n(c) = \frac{1}{1-c} - \log\frac{c}{1-c}+\frac{1-2\bar{c}}{2\bar{c}^2(1-\bar{c})^2}c^2
\end{equation}
so that $l_{\alpha_2}'(c_1) = \log(p/(1-p))+ n(c_1)$ and $l_{\alpha_2}'(c_2) =\log(p/(1-p)) + n(c_2)$. We have $n(0) = \infty$, $n(1) = \infty$, with derivative $n'(c)$ given by
\begin{equation}
			n'(c) = c\bigg ( \frac{1-2\bar{c}}{\bar{c}^2(1-\bar{c})^2} - \frac{1 -  2c}{c^2(1-c)^2}\bigg) = c\left(l_{\alpha_2}'''(\bar{c}) - l_{\alpha_2}'''(c)\right).
\end{equation}
As $l_{\alpha_2}'''(c)$ is monotonically decreasing on $[0, 1]$, $n(c)$ is decreasing from $0$ to $\bar{c}$, and increasing from $\bar{c}$ to $1$, with the global minimum achieved at $\bar{c}$,
\begin{equation}
			n(\bar{c}) = \frac{1}{1-\bar{c}} - \log\frac{\bar{c}}{1-\bar{c}} + \frac{1-2\bar{c}}{2(1-\bar{c})^2}.
\end{equation}
This implies that $l_{\alpha_2}'(c_1; p, \alpha_1^c) \ge 0 $ for $p\geq p^c$ (\ref{p}). The only possible region in the $(p,\alpha_1)$-plane where $l_{\alpha_2}'(c_1) < 0 < l_{\alpha_1}'(c_2)$ is thus bounded by $p < p^c$  and $\alpha_1 > \alpha_1^c$.

Finally, we analyze the behavior of $l_{\alpha_1}'(c_1)$ and $l_{\alpha_1}'(c_2)$ more closely when $p$ and $\alpha_1$ are chosen from this region. Recall that $c_1 < \bar{c} < c_2$. By monotonicity of $n(c)$ on the intervals $(0, \bar{c})$ and $(\bar{c} , 1)$, there exist continuous functions $a(p)$ and $b(p)$ of $p$, such that $l_{\alpha_2}'(c_1) < 0$ for $c_1 > a(p)$ and
		$l_{\alpha_2}'(c_2) > 0$ for $c_2 > b(p)$. As $p \to 0$, $a(p) \to 0$ and $b(p) \to 1$. $a(p)$ is an increasing function of $p$, whereas $b(p)$ is a decreasing function, and they satisfy $n(a(p)) = n(b(p)) = - p$. The restrictions on $c_1$ and $c_2$ yield restrictions on $\alpha_1$, and we have $l_{\alpha_2}'(c_1) < 0$ for $\alpha_1 < m(a(p))$ and $l_{\alpha_2}'(c_2) > 0$ for $\alpha_1 > m(b(p))$. As $p \to 0, m(a(p)) \to \infty$ and $m(b(p)) \to \infty$. $m(a(p))$ and $m(b(p))$ are both decreasing functions of $p$ and they satisfy $l_{\alpha_2}'(c_1; p, m(a(p)) = l_{\alpha_2}'(c_2; p, m(b(p)) = 0$. As $l_{\alpha_2}'(c_2; p, \alpha_1) > l_{\alpha_2}'(c_1; p,\alpha_1)$ for every $(p, \alpha_1)$, the curve $m(b(p))$ must lie below the curve $m(a(p))$, and together they generate the bounding curves for the V-shaped region in the $(p,\alpha_1)$-plane with corner point $(p^c, \alpha_1^c)$ where two local maximizers exist for $l_{\alpha_2}(c)$.		
\end{proof}

From Proposition \ref{max}, the critical curve is traced out by $(p^c, \alpha_1^c, \alpha_2)$ (\ref{p}) (\ref{alpha1}) (\ref{alpha2}), where we take $1/2\leq \bar{c}=c^* \leq 2/3$ to meet the non-negativity constraints on $\alpha_1^c$ and $\alpha_2$. See Figure \ref{3D}. We delve deeper into the behavior of the function $\lambda$ (\ref{lamb}) along the critical curve. To lighten the notation, we denote the associated parameters by $(p, \alpha_1, \alpha_2)$, and the unique local maximizer by $c^*$ with $1/2\leq c^*\leq 2/3$.

\begin{lemma}
\label{burnin}
Along the critical curve, we have

\noindent (1) $\lambda(c^*)=c^*$, $\lambda'(c^*)=1$, $\lambda''(c^*)=0$, and $\lambda'''(c^*)\leq -8$.

\noindent (2) $\lambda''(c)\leq 0$ for $c\geq c^*$, and $\lambda''(c)\geq 0$ for $c\leq c^*$.
\end{lemma}

\begin{proof}
The first claim follows from direct computation. We spell out some details. It is clear that $\lambda(c^*) = c^*$. Letting
$A = p\exp(\alpha_1 c + \frac{\alpha_2}{2} c^2)/(1-p)$, we write $\lambda$'s first few derivatives as
\begin{equation*}
\lambda'(c) = \frac{\alpha_1 + \alpha_2 c}{1+A} \lambda(c),
\end{equation*}
\begin{equation*}
\lambda''(c) = \lambda(c) \frac{\alpha_2}{1+A} + (1-A) \frac{\lambda'(c)^2}{\lambda(c)},
\end{equation*}
\begin{equation}
\lambda'''(c) = \lambda'(c) \frac{\alpha_2}{1+A} - \alpha_2 \lambda(c) \lambda'(c) + 2(1-A)\frac{\lambda'(c)\lambda''(c)}{\lambda(c)} - (1-A)\frac{\lambda'(c)^3}{\lambda(c)^2} - (\alpha_1 + \alpha_2 c)A \frac{\lambda'(c)^2}{\lambda(c)}.
\end{equation}
Substituting the parameter values (\ref{p}) (\ref{alpha1}) (\ref{alpha2}) yields $\lambda'(c^*) = 1$, $\lambda''(c^*) = 0$, and
\begin{equation}
\lambda'''(c^*) = \frac{-6(c^*)^2 + 6(c^*) - 2}{(c^*)^2 (1-c^*)^2}\leq -8.
\end{equation}

For the second claim, we show that for $c \ge c^{\ast}$, $\lambda''(c) \le 0$. The parallel claim may be verified using a similar line of reasoning. As derived previously,
\begin{equation}
\lambda''(c) = \frac{A}{(1+A)^3} \left( (1+A)\alpha_2 + (1-A)(\alpha_1 + \alpha_2 c)^2 \right).
\end{equation}
Notice that $\lambda''(c)\leq 0$ precisely when
\begin{equation}
\frac{\alpha_2 + (\alpha_1 + \alpha_2 c)^2 }{(\alpha_1 + \alpha_2 c)^2 - \alpha_2} \le A,
\end{equation}
where equality holds when $c = c^{\ast}$. For $c$ increasing from $c^{\ast}$, $A$ is increasing whereas the left hand of the above inequality is decreasing. Our claim thus follows.
\end{proof}

By Lemma \ref{burnin}, the expected magnetization drift $(\lambda(c^*)-c^*)/n$ (\ref{expTstep}) drops from first order to third order along the critical curve as compared with other parameter regions. As a result, we anticipate that the burn-in will be slower. The following Theorem \ref{last} establishes an upper bound. Utilizing coupling techniques from Levin et al. \cite{LLP} with minor adaptation, an $O(n^{2/3})$ mixing of the Glauber dynamics is further expected. For details, see Lemma 2.9 and Theorem 4.1 of \cite{LLP}.

\begin{theorem}
\label{last}
Along the critical curve, the burn-in time for the Glauber dynamics is $O(n^{3/2})$.
\end{theorem}

\begin{proof}
From Lemma \ref{expmagnetization},
\begin{equation}
\ER(c_{t+1} - c^{\ast} \mid c_t - c^{\ast}) = \frac{1}{n} \left( \lambda(c_t) - \lambda(c^{\ast}) \right) + \left(1 - \frac{1}{n}\right) (c_t - c^{\ast}) + O\left(\frac{1}{n^2}\right).
\end{equation}
Let $e_t = c_t - c^{\ast}$ and define $g(e_t) = \lambda(e_t + c^{\ast}) - \lambda(c^{\ast})$. Then
\begin{equation*}
\ER(|e_{t+1}| \mid e_t) = \frac{1}{n} g(|e_t|) + \left(1 - \frac{1}{n}\right) |e_t| + O\left(\frac{1}{n^2}\right) \text{ for } e_t \geq 0,
\end{equation*}
\begin{equation}
\label{E}
\ER(|e_{t+1}| \mid e_t) = -\frac{1}{n}g(-|e_t|) + \left(1 - \frac{1}{n}\right) |e_t| + O\left(\frac{1}{n^2}\right) \text{ for } e_t<0.
\end{equation}
Define $\tau_0 = \min\{ t \ge 0 : |e_t| \le 1/n \}$. Note that $e_t$ does not change sign when $t<\tau_0$. Multiplying both sides of (\ref{E}) by the indicator function $\mathbf{1}_{\{\tau_0 > t\}}$ and using that $g(0) = 0$ and $\mathbf{1}_{\{\tau_0 > t + 1\}} \le \mathbf{1}_{\{\tau_0 > t\}}$, we have
\begin{equation*}
		\ER\left( |e_{t+1}| \mathbf{1}_{\{\tau_0 > t + 1\}} \mid e_t \right) \le \frac{1}{n}g\left( |e_t| \mathbf{1}_{\{\tau_0 > t\}} \right) + \left(1-\frac{1}{n}\right)|e_t| \mathbf{1}_{\{\tau_0 > t\}} + O(\frac{1}{n^2}),
	\end{equation*}
\begin{equation}
\label{Etheta}
\ER(|e_{t+1}|\mathbf{1}_{\{\tau_0 > t+1\}} \mid e_t) \leq -\frac{1}{n}g(-|e_t|\mathbf{1}_{\{\tau_0 > t\}}) + \left(1 - \frac{1}{n}\right) |e_t|\mathbf{1}_{\{\tau_0 > t\}} + O\left(\frac{1}{n^2}\right).
\end{equation}
Let $\theta_t^{+} = \ER\left( |e_t| \mathbf{1}_{\{ \tau_0 > t \}} \right)$. By Lemma \ref{burnin}, $g(e)\leq 0$ for $e\geq 0$ and $g(e)\geq 0$ for $e\leq 0$, $g$ is concave down on the non-negative axis and concave up on the negative axis. Taking expectation of both sides of (\ref{Etheta}) and applying Jensen's inequality on $g$,
\begin{equation*}
\theta_{t+1}^{+} - \theta_{t}^{+} \le - \frac{1}{n} \left( \theta_t^{+} - g(\theta_t^{+}) \right) + O(\frac{1}{n^2}),
\end{equation*}
\begin{equation}
\theta_{t+1}^{+} - \theta_{t}^{+} \le - \frac{1}{n} \left( g(-\theta_t^{+})-\left(-\theta_t^{+}\right)\right) + O(\frac{1}{n^2}).
\end{equation}

Let $\mu > 0$ and suppose that $\theta_t^{+} \ge \mu$. As in the proof of Theorem \ref{negdrift}, by the extreme value theorem, there exists $\gamma(\mu) > 0$ such that
\begin{equation}
\theta_{t+1}^{+} - \theta_t^{+} \le - \frac{\gamma(\mu)}{n}.
\end{equation}
Utilizing the negative drift $-\gamma(\mu)/n$, there exists a time $t^*=O(n)$ so that $\theta_t^{+} \leq 1/4$ for all $t\geq t^*$. Consider the Taylor series expansion of $g(\theta_t^{+})$ and $g(-\theta_t^{+})$ and using Lemma \ref{burnin}:
\begin{equation*}
		g(\theta_t^{+}) = \lambda'(c^{\ast}) \theta_t^{+} + \frac{\lambda''(c^{\ast})}{2}(\theta_t^{+})^2 + \frac{\lambda'''(d_1)}{6}(\theta_t^{+})^3
			\leq \theta_t^{+} - \frac{4}{3}(\theta_t^{+})^3,
\end{equation*}
\begin{equation}
		g(-\theta_t^{+}) = -\lambda'(c^{\ast}) \theta_t^{+} + \frac{\lambda''(c^{\ast})}{2}(\theta_t^{+})^2 - \frac{\lambda'''(d_2)}{6}(\theta_t^{+})^3
			\geq -\theta_t^{+} + \frac{4}{3}(\theta_t^{+})^3,
\end{equation}
where $d_1 \in [c^{\ast}, c^*+\theta_t^{+}]$ and $d_2\in [c^*-\theta_t^+, c^*]$. Then it follows that
\begin{equation}
\theta_{t+1}^{+} \le \theta_t^{+} - \frac{4}{3n}(\theta_t^{+})^3 + O(\frac{1}{n^2})
\end{equation}
for $t\geq t^*$. The remainder of the proof follows analogously as in the proof of Theorem 4.1 in \cite{LLP}. For some $c$,
\begin{equation}
\lim_{c \to \infty} \PR \left( \tau_0 > cn^{3/2} \right) = 0
\end{equation}
uniformly in $n$, which implies that the Glauber dynamics may be coupled so that the magnetizations agree in $O(n^{3/2})$ time steps.
\end{proof}

\section{Generalizations and future work}
\label{generalizations}
Numerous extensions can be made about these vertex-weighted exponential models. We started the discussion with the edge-triangle lattice gas (Ising) model, but clearly more complicated subgraph densities can be considered. Denote by $K_k$ a complete graph on $k$ vertices so that an edge is $K_2$ and a triangle is $K_3$. For consistency, denote by $C_0 = 1$, $C_1 = S(X, i)$, and $C_2 = T(X, i)$. As in (\ref{ST}), the crucial quantity
\begin{equation}
C_m(X, i) = \sum_{i\neq i_1 \neq i_2 \neq \dots \neq i_m} X_{i_1} X_{i_2} \cdots X_{i_m}
\end{equation}
satisfies $C_m = {C_1 \choose m}$, which may be justified by noting that a term in the sum defining $C_m$ has value $1$ if and only if every vertex in the product has spin $1$. Moving on, we have employed a discrete-time update of the network, but the network may be updated on a continuous-time basis, and this may be realized by posing iid Poisson clocks and examining the corresponding heat kernel. More significantly, rather than the simplifying assumption that a person is either interested or not in building a friendship, in reality a person probably has different levels of interest in forming a connection, and an edge is placed between two people when the joint interest exceeds a certain threshold value. Also, in social networks people have diverse attributes; only people with the same attribute or those with more than a specified number of attributes will establish a tie, which will fall within the regime of the random cluster model and multilayer networks. All these extensions are quite challenging both theoretically and computationally, especially when network geometry comes into play, but we hope to address at least some of them in future work.

After we gain an understanding of the small-world observed structure of big network data using Markov chain dynamics, we may use this knowledge for the prediction and control of general spreading processes on large-scale networks. These processes include the social influence of opinions, users' decisions to adopt products, and epidemic intervention strategies, etc. To illustrate, we return to information diffusion over Twitter mentioned at the beginning of this paper. Updating the weight of a vertex corresponding to a celebrity or a news source will definitely have more impact than that for ordinary people. So instead of running the Glauber dynamics that chooses a vertex at random, we choose the ``hubs'' of the network to update. This selective procedure decreases the mixing time and drives the spreading dynamics more efficiently towards equilibrium. Other properties of the chains may be studied simultaneously. For example, the cover time of the network may be interpreted as a realization of a ``web crawl'', and the hitting time may be interpreted as the necessary local queries to determine the global connectivity.

\section*{Acknowledgements}

Mei Yin thanks Richard Kenyon for helpful conversations.

\end{document}